\documentclass[10pt, a4paper,
oneside, headinclude,footinclude]{scrartcl}
\usepackage[utf8]{inputenc}

\usepackage[
nochapters, 
pdfspacing, 
dottedtoc 
]{classicthesis} 
\usepackage{amsopn}
\usepackage[T1]{fontenc} 
\usepackage{multirow}
\usepackage[utf8]{inputenc} 
\usepackage{hhline}
\usepackage{graphicx} 
\graphicspath{{Figures/}} 
\usepackage{indentfirst}
\usepackage{enumitem} 
\usepackage{savesym}
\usepackage{mathrsfs}
\usepackage{geometry}
\usepackage{esint}
\usepackage{longtable}

\usepackage[
    backend=biber,
    style=numeric,
    natbib=false,
    url=false, 
    doi=false,
    eprint=false,
    maxnames=50
]{biblatex}

\usepackage{stmaryrd}

\usepackage{subfig} 

\usepackage{amsmath, amssymb,amsthm,amsfonts} 

\usepackage[english,capitalize]{cleveref}

\usepackage{thmtools}
\usepackage{thm-restate}

\theoremstyle{plain}
\newtheorem{teorema}{Theorem}[section]
\newtheorem{proposizione}[teorema]{Proposition}
\newtheorem{lemma}[teorema]{Lemma}

\newtheorem*{theorem*}{Theorem}

\theoremstyle{definition}
\newtheorem{definizione}{Definition}[section]

\theoremstyle{remark}

\newcommand{\Mass}{\mathbb{M}}
\newcommand{\R}{\mathbb{R}}
\newcommand{\N}{\mathbb{N}}

\newcommand{\im}{\mathrm{im}}
\newcommand{\Haus}{\mathscr{H}}
\newcommand{\Leb}{\mathscr{L}}

\newcommand{\Tan}{\mathrm{Tan}}

\newcommand{\supp}{\mathrm{supp}}
\newcommand{\Gr}{\mathrm{Gr}}

\newcommand{\dV}{d_V\kern-1pt}
\newcommand{\dW}{d_W\kern-1pt}

\newcommand{\trait}[3]{\vrule width #1ex height #2ex depth #3ex}
\newcommand{\trace}{\mathchoice%
  {\mathbin{\trait{.12}{1.2}{.03}\trait{.8}{0.09}{0.03}}}
  {\mathbin{\trait{.12}{1.2}{.03}\trait{.8}{0.09}{0.03}}}
  {\mathbin{\hskip.15ex\trait{.09}{.84}{0.02}\trait{.56}{.07}{.02}}\hskip.15ex}
  {\mathbin{\trait{.07}{.6}{.01}\trait{.4}{.06}{.01}}}}

\newcounter{const}

\newcounter{eps}

\newcommand{\vertiii}[1]{{\left\vert\kern-0.25ex\left\vert\kern-0.25ex\left\vert #1 
    \right\vert\kern-0.25ex\right\vert\kern-0.25ex\right\vert}}

\hypersetup{
colorlinks=true, breaklinks=true,bookmarksnumbered,
urlcolor=webbrown, linkcolor=RoyalBlue, citecolor=webgreen, 
pdftitle={}, 
pdfauthor={\textcopyright}, 
pdfsubject={}, 
pdfkeywords={}, 
pdfcreator={pdfLaTeX}, 
pdfproducer={LaTeX with hyperref and ClassicThesis} 
} 

\title{\normalfont\spacedallcaps{Characterization of rectifiability via Lusin type approximation}} 
\author{\spacedlowsmallcaps{Andrea Marchese \textsuperscript{*} and Andrea Merlo\textsuperscript{**}}}

\date{}

\begin{document}

\renewcommand{\sectionmark}[1]{\markright{\spacedlowsmallcaps{#1}}} 
\lehead{\mbox{\llap{\small\thepage\kern1em\color{halfgray} \vline}\color{halfgray}\hspace{0.5em}\rightmark\hfil}} 
\pagestyle{scrheadings}
\maketitle 
\setcounter{tocdepth}{2}

{\let\thefootnote\relax\footnotetext{* \textit{Università di Trento, Dipartimento di Matematica, Via Sommarive, 14, 38123 Povo (TN), Italy.}}}
{\let\thefootnote\relax\footnotetext{** \textit{Université Paris-Saclay, 307 Rue Michel Magat Bâtiment, 91400 Orsay, France.}}}

{\rightskip 1 cm
\leftskip 1 cm
\parindent 0 pt
\footnotesize

	%
{\textsc Abstract.}
We prove that a Radon measure $\mu$ on $\mathbb{R}^n$ can be written as $\mu=\sum_{i=0}^n\mu_i$, where each of the $\mu_i$ is an $i$-dimensional rectifiable measure if and only if for every Lipschitz function $f:\mathbb{R}^n\to\mathbb{R}$ and every $\varepsilon>0$ there exists a function $g$ of class $C^1$ such that $\mu(\{x\in\mathbb{R}^n:g(x)\neq f(x)\})<\varepsilon$.   
\par
}

\section{Introduction}
A fundamental yet simple consequence of Rademacher's theorem and Whitney's theorem is the fact that Lipschitz functions on the Euclidean space admit a Lusin type approximation with $C^1$-functions, namely for every Lipschitz function $f:\R^n\to\R$ and every $\varepsilon>0$ there exists a function $g:\R^n\to\R$ of class $C^1$ such that
$$\Leb^n(\{x\in\R^n:g(x)\neq f(x)\})<\varepsilon,$$
where $\Leb^n$ denotes the Lebesgue measure, see \cite[Theorem5.3]{Simon1983LecturesTheory}. This fact has a central role in many basic results in Geometric Measure Theory, including the existence of the approximate tangent space to a rectifiable set, see \cite[Lemma 11.1]{Simon1983LecturesTheory}, and the validity of area and coarea formulas, see \cite[\S 12]{Simon1983LecturesTheory}.\\

On one side, this approximation property does not only hold for the Lebesgue measure: for instance it holds trivially for a Dirac delta. It is not difficult to see that the same property holds for any rectifiable measure and clearly the class of Radon measures for which the property holds is closed under finite sums.

On the other side, it is known that there are measures $\mu$ for which Lipschtz functions do not admit a Lusin type approximation with respect to $\mu$ with functions of class $C^1$, see \cite{Mar}. In this note we go beyond such result, proving that the validity of such approximation property characterizes rectifiable measures, in the following sense.

\begin{teorema}\label{t:main}
Let $\mu$ be a positive Radon measure on $\R^n$. The measure $\mu$ can be written as $\mu=\sum_{i=0}^n\mu_i$, where each of the $\mu_i$ is an $i$-dimensional rectifiable measure if and only if for every Lipschitz function $f:\R^n\to\R$ and every $\varepsilon>0$ there exists a function $g$ of class $C^1$ such that $$\mu(\{x\in\R^n:g(x)\neq f(x)\})<\varepsilon.$$
\end{teorema}
The proof of the "only if" part of Theorem \ref{t:main} is a simple application of Whitney's theorem. The proof of the "if" part exploits some tools introduced in \cite{AlbMar}, including the notion of \emph{decomposability bundle} of a measure $\mu$, see \cite[\S 2.6]{AlbMar}: a map $x\mapsto V(\mu,x)$ which detects the maximal subspaces along which Lipschitz functions are differentiable $\mu$-almost everywhere. For the purposes of this paper, we need to refine the result \cite[Theorem 1.1 (ii)]{AlbMar} on the existence of Lipschitz functions which are non-differentiable along directions which do not belong to the decomposability bundle. In \cite{AlbMar}, such non-differentiability is proved by finding a Lipschitz function $f$ and for $\mu$-almost every point $x$ a sequence of points $y_i:=x+t_iv\in\R^n$ converging to $x$ along a direction $v\not\in V(\mu,x)$, such that the corresponding incremental ratios $(f(y_i)-f(x))/t_i$ do not converge. Here we need to find a function $f$ such that there exist points $y_i$ as above, with the additional requirement that $y_i\in\supp(\mu)$, see Proposition \ref{coroll}. For a non rectifiable measure $\mu$, the existence of a $\mu$-positive set of points $x$ for which there are points $y_i\in\supp(\mu)$ approaching $x$ along a direction $v\not\in V(\mu,x)$ is guaranteed by Lemma \ref{lemma:coni}.\\  

It is worth noting that Lusin-type approximation theorems are an interesting tool to study even in metric measure spaces. In \cite{julia2021lipschitz} the authors proved a suitable extension of Lusin's approximation-type theorem for the surface measure of $1$-codimensional $C^1_\mathbb{H}$-rectifiable surfaces in the Heisenberg groups $\mathbb{H}^n$, $n\geq 2$ and where the regular approximation of Lipschitz functions are found in the class of $C^1_\mathbb{H}$-regular functions. The authors also prove that in $\mathbb{H}^1$ there is a regular surface and a Lipschitz function that cannot be approximated by $C^1_\mathbb{H}$-regular functions. 
This different behaviour is connected to the algebraic structure of the tangents to $1$-codimensional regular surfaces in the Heisenberg groups $\mathbb{H}^n$ when $n=1$ or $n\geq 2$. This will be object of further investigation in general Carnot groups thanks to the techniques developed in \cite{reversepansu}.

\section{Notation and preliminaries}
We denote by $U(x,r)$ the open ball in $\R^n$ with center $x$ and radius $r$ and by $B(x,r)$ the closed ball. In addition, for a Borel set $E$ and a $\delta>0$, we denote $B(E,\delta):=\bigcup_{y\in E}B(y,\delta)$. The unit sphere is denoted $\mathbb{S}^{n-1}$. 

Given a Radon measure $\mu$ and a (possibly vector-valued) function $f$, we denote by $f\mu$ the measure $$f\mu(A):=\int_Af d\mu,\quad \mbox{for every Borel set $A$}.$$
For a measure $\mu$ and a Borel set $E$ we denote by $\mu\trace E$ the restriction of $\mu$ to $E$, namely the measure defined by 
$$\mu\trace E(A):=\mu (A\cap E), \quad \mbox{for every Borel set $A$}.$$
The support of a positive Radon measure $\mu$, denoted $\supp(\mu)$, is the intersection of every closed sets $C$ such that $\mu(\R^n\setminus C)=0$. 
For $0\leq k\leq n$, the symbol $\Haus^k$ denotes the $k$-dimensional Hausdorff measure on $\R^n$.

\begin{definizione}[Rectifiable sets and measures]
For $0\leq k\leq n$, a set $E\subset\R^n$ is $k$\emph{-rectifiable} if there are sets $E_i$ ($i=1,2,\ldots$) such that
\begin{itemize}
    \item [(i)] $E_i$ is a Lipschitz image of $\R^k$ for every $i$;
    \item [(ii)] $\Haus^k(E\setminus\bigcup_{i\geq 1}E_i)=0$. 
\end{itemize}
A Radon measure is said to be $k$\emph{-rectifiable} if  it is absolutely continuous with respect to $\Haus^k\trace E$, for some $k$-rectifiable set $E$. 
\end{definizione}

AS usual, the symbol $\Gr(k,n)$ denotes the Grassmannian of $k$-planes in $\R^n$, and we define $\Gr:=\bigcup_{0\leq k \leq n}\Gr(k,n)$. We endow $\Gr$ with the topology induced by the distance
$$d(V,W):=d_\Haus(V\cap U(0,1), W\cap U(0,1)),$$
where $d_\Haus$ is the Hausdorff distance. We recall the following definition, see \cite[\S 2.6, \S 6.1 and Theorem 6.4]{AlbMar}. 

\begin{definizione}[Decomposability bundle]
Given a positive Radon measure $\mu$ on $\R^n$ its \emph{decomposability bundle} is a map $V(\mu,\cdot)$ taking values in the set $\Gr$ defined as follows. A vector $v\in\R^n$ belongs to $V(\mu,x)$ if and only if there exists a vector-valued measure $T$ with ${\mathrm {div}} T=0$ such that
$$\lim_{r\to 0}\frac{\Mass((T-v\mu)\trace B(x,r))}{\mu(B(x,r))}=0,$$
where $\Mass((T-v\mu)\trace B(x,r))$ denotes the total variation of the vector-valued measure $(T-v\mu)\trace B(x,r)$.
\end{definizione}

\begin{definizione}[Tangent measures]
 We define the map $T_{x,r}(y)=\tfrac{y-x}{r}$, and we denote by $T_{x,r}\mu$ the pushforward of $\mu$ under $T_{x,r}$, namely $T_{x,r}\mu(A):=\mu(x+rA)$ for every Borel set $A$. Given a measure $\mu$ and a point $x$, the family of \textit{tangent measures}  $\Tan(\mu,x)$, introduced in \cite{Preiss1987GeometryDensities}, consists of all the possible non-zero limits (with respect to the weak* convergence of measures) of $c_iT_{x,r_i}\mu$, for some sequence of positive real numbers $c_i$ and some sequence of radii $r_i\to 0$. We know thanks to \cite[Theorem 2.5]{Preiss1987GeometryDensities} that $\Tan(\mu,x)$ is non-empty $\mu$-almost everywhere.
\end{definizione}

\begin{definizione}[Cone over a $k$-plane]
For any $k\in\{1,\ldots,n-1\}$, $0<\alpha<1$, $x\in\R^n$ and $V\in\Gr(k,n)$ we let:
$$X(x,V,\alpha):=x+\{v\in\R^n:|p_V(v)|\geq\alpha|v|\},$$
where $p_V$ denotes the orthogonal projection onto $V$.
For notation convenience, for $k=0$ and for every $0<\alpha<1$, we define $X(x,0,\alpha):=\{x\}$.
\end{definizione}

\begin{definizione}[Distance $F_K$ between measures]\label{def:Fk}
Given $\phi$ and $\psi$ two Radon measures on $\R^n$, and given $K\subseteq \R^n$ a compact set, we define 
\begin{equation}
    F_K(\phi,\psi):= \sup\left\{\left|\int fd\phi - \int fd\psi\right|:f\in \mathrm{Lip}_1^+(K)\right\},
    \label{eq:F}
\end{equation}
where $\mathrm{Lip}_1^+(K)$ denotes the class of $1$-Lipschitz nonnegative function with support contained in $K$. We also write $F_{x,r}$ for $F_{B(x,r)}$.
\end{definizione}

\begin{lemma}\label{lemma:coni}
Let $\mu$ be a Radon measure on $\R^n$ with ${\dim}(V(\mu,x))=k<n$, for $\mu$-almost every $x$. Assume that $\mu(R)=0$ for every $k$-rectifiable set $R$. Then for $\mu$-almost every $x$ there exists $\alpha>0$ such that for every $\varepsilon>0$ 
\begin{equation}\label{tantipunti}
{\supp}(\mu)\cap B(x,\varepsilon)\setminus    X(x,V(\mu,x),\alpha)\neq\emptyset.
\end{equation} 
\end{lemma}

\begin{proof}
Assume by contradiction that there exists a Borel set $E$ with $\mu(E)>0$ such that for every $x\in E$ and for every $\alpha>0$ there exists $\varepsilon>0$ such that \eqref{tantipunti} fails. We claim that this implies that for $\mu$-almost every $x\in E$ and for every $\alpha>0$ every tangent measure $\nu\in \Tan(\mu,x)$ satisfies
\begin{equation}\label{blowupinconi}
    \supp(\nu)\subset  X(0,V(\mu,x),\alpha)
\end{equation} and therefore  $\supp(\nu)\subset \bigcap_{\alpha>0}X(0,V(\mu,x),\alpha)=V(\mu,x)$. 
In order to prove \eqref{blowupinconi}, fix $x\in E$ such that $\Tan(\mu,x)$ is non empty and consider any open ball $U(y,\rho)\subset \R^n\setminus X(0,V(\mu,x),\alpha)$ and notice that since \eqref{tantipunti} fails, we have $T_{x,r}\mu(U(y,\rho))=\mu(U(x+ry,r\rho))=0$ for every $r<\varepsilon/(|y|+\rho)$ which concludes in view of \cite[Proposition 2.7]{DeLellis2008RectifiableMeasures}. Thanks to \cite[Proposition 2.9]{delnin2021endpoint} we infer in particular that $\nu=c\Haus^k\trace V(\mu,x)$ for some $c>0$.
For every $W\in\Gr(k,n)$ denote
$$E_W:=\{x\in\R^n:(k+1)F_{0,1}(\Haus^k\trace V(\mu,x),\Haus^k\trace W)<20^{-k-4}\}.$$
By the compactness of the Grassmannian, there exists $W\in\Gr(k,n)$ such that $\mu(E_W)>0$. On the other hand, by \cite[\S 4.4(5)]{Preiss1987GeometryDensities} and by the locality of tangent measures, see \cite[\S 2.3(4)]{Preiss1987GeometryDensities}, we conclude that $\mu\trace E_W$ is supported on a $k$-rectifiable set. This however contradicts the assumption that $\mu(R)=0$ for every $k$-rectifiable set $R$.
\end{proof}

\begin{definizione}[Cone-null sets]
 For any $e\in\mathbb{S}^{n-1}$ and $\alpha\in(0,1)$ we let the \emph{one-sided cone of axis $e$ and amplitude $\alpha$} the set
$$C(e,\alpha):=\{v\in\R^n:\langle v,e\rangle\geq\alpha|v|\}.$$
In the following we denote by $\Gamma(e,\alpha)$ the family of Lipschitz curves $\gamma:E\subseteq \R\to\R^n$ such that $\gamma'(t)\in C(e,\alpha)$ for $\mathcal{L}^1$-almost every $t\in E$. 
Finally, a Borel set $B$ is said to be $C(e,\alpha)$-null if $\Haus^1(\mathrm{im}(\gamma)\cap B)=0$ for any $\gamma\in \Gamma(e,\alpha)$.
\end{definizione}

\begin{proposizione}\label{alb}
Let $E$ be a compact set in $\R^n$.
Let $W\in\Gr(k,n)$, with $k<n$ and suppose that there exists $\theta_0\in(0,1)$ such that for any $e\in W^\perp$ the set $E$ is $C(e,\theta_0)$-null. Then, for any $\theta_0\leq \theta<1$  and $\varepsilon>0$ there exists $\delta_0>0$ such that
$$\Haus^1(\im(\gamma)\cap B(E,\delta_0))\leq \varepsilon,$$
for any $\gamma\in\Gamma(e,\theta)$.
For any $\theta_0\leq \theta<1$, $0<\delta<\delta_0$ and any $e\in W^\perp$, consider the function
\begin{equation}
    \omega_{e,\theta,\delta}(x):=\sup_{\gamma\in \Gamma(e,\theta)}\Haus^1(B(E,\delta)\cap \im(\gamma))-\lambda\lvert e\rvert,
    \label{e:width}
\end{equation}
where $\gamma(b)$ is of the form $x+\lambda e$. Then the following properties hold
\begin{itemize}
\item[(i)] $0\leq \omega_{e,\theta,\delta}(x)\leq \varepsilon$ for any $x\in\R^n$,
\item[(ii)]  $\omega_{e,\theta,\delta}(x)\leq \omega_{e,\theta,\delta}(x+se)\leq \omega_{e,\theta,\delta}(x)+s|e|$ for every $s>0$ and any $x\in\R^n$. Moreover, if the segment $[x,x+se]$ is contained in $B(E,\delta)$, then
$\omega_{e,\theta,\delta}(x+se)=\omega_{e,\theta,\delta}(x)+s\lvert e \rvert$,
\item[(iii)] $\lvert \omega_{e,\theta,\delta}(x+v)-\omega_{e,\theta,\delta}(x)\rvert\leq \theta(1-\theta^2)^{-1/2}\lvert v\rvert$ for every $v\in V:=e^{\perp}$,
\item[(iv)]$\omega_{e,\theta,\delta}$ is $1+(n-1)\theta(1-\theta^2)^{-1/2}$-Lipschitz.
\end{itemize}
\end{proposizione}

\begin{proof}
The first part of the proposition is an immediate consequence of \emph{Step 1} in the proof of  \cite[Proposition 4.12]{AlbMar}.
On the other hand, the construction of the function $\omega_{e,\theta,\delta}$ was performed in the second step of the proof of \cite[Proposition 4.12]{AlbMar}.
\end{proof}

\section{Construction of non-differentiable functions}
Throughout this section we fix $k\in \{0,\ldots,n-1\}$ and let $\mu$ be a Radon measure such that $\dim(V(\mu,x))=k$ for $\mu$-almost every $x\in\R^n$ and that $\mu(R)=0$ for any $k$-rectifiable set $R$.
Thanks to the strong locality principle, see \cite[Proposition 2.9 (i)]{AlbMar}, and Lusin's Theorem we can assume, up to restriction to a compact subset $\tilde K\subset\supp(\mu)$ of positive $\mu$-measure, that $V(\mu,x)$ is uniformly continuous on $\tilde K$.

\medskip

The aim of this section is to prove the following

\begin{proposizione}\label{coroll}
Let $\mu$ and $\tilde K$ be as above. There exists a Lipschitz function $f:\R^n\to \R$ and a Borel set $E\subseteq \tilde{K}$ of positive $\mu$-measure such that for $\mu$-almost every $x\in E$ there exists a direction $v\not\in V(\mu,x)$ and a sequence of points $y_i=y_i(x)\in \tilde K$ such that
    $$\frac{y_i-x}{\lvert y_i-x\rvert}\to v\qquad \text{and}\qquad \limsup_{i\to \infty}\frac{f(y_i)-f(x)}{\lvert y_i-x\rvert}-\liminf_{i\to \infty}\frac{f(y_i)-f(x)}{\lvert y_i-x\rvert}>0.$$
\end{proposizione}

Thanks to Lemma \ref{lemma:coni} we know that there exists $0<\alpha_0<1/\sqrt{n}$ such that for any $0<\alpha<\alpha_0$ we can find a compact subset $K_\alpha$ of $\tilde K$ with positive measure where
\begin{equation}\label{tantipunti2}
{\supp}(\mu)\cap B(x,r)\setminus    X(x,V(\mu,x),\sqrt{1-\alpha^2})\neq\emptyset\qquad\text{for any $r>0$ and every $x\in K_\alpha$}.
\end{equation}

\begin{lemma}\label{lemma:campo}
Let $\mu$ and $(K_\alpha)_{0<\alpha<\alpha_0}$ be as above. For any $0<\alpha<\alpha_0$ we can find a compact set $K\subseteq K_\alpha$ of positive $\mu$-measure and a continuous vector field $e:\R^n\to\mathbb{S}^{n-1}$ such that $e(x)$ is orthogonal to $V(\mu,x)$ at $\mu$-almost every $x\in\R^n$ and such that:
\begin{equation}\label{e:conipieni}
\supp(\mu)\cap B(x,r)\cap C(e(x),(n-k)^{-1}\alpha)\setminus    X(x,V(\mu,x),\sqrt{1-\alpha^2})\neq \emptyset \qquad \text{for any $r>0$ and for every $x\in K$}.
\end{equation}
\end{lemma}

\begin{proof}
Thanks to the continuity of $V(\mu,\cdot)^\perp$ we can find $n-k$ continuous vector fields  $e_{k+1},\ldots,e_{n}:\R^n\to\mathbb{S}^{n-1}$ such that 
$$V(\mu,x)^\perp=\mathrm{span}\{e_{k+1}(x),\ldots,e_{n}(x)\},$$
for every $x\in K_\alpha$. Since the cones $$C(e_{k+1}(x),(n-k)^{-1}\alpha),\ldots,C(e_{n}(x),(n-k)^{-1}\alpha),C(-e_{k+1}(x),(n-k)^{-1}\alpha),\ldots,C(-e_{n}(x),(n-k)^{-1}\alpha),$$ cover $\R^n \setminus X(0,V(\mu,x),\sqrt{1-\alpha^2})$ for every $x\in K_\alpha$, there exists one vector field, that we denote $e$, among the $e_{k+1},\ldots,e_{n},-e_{k+1},\ldots,-e_{n}$ for which the set of those $x\in K_\alpha$ where
\eqref{e:conipieni} holds has positive $\mu$-measure. 
\end{proof}

\begin{definizione}
Throughout the rest of this section we will let $\alpha_0$ be as in \eqref{tantipunti2} and we fix $0<\alpha<\alpha_0$. We also fix the compact set $K$ and the continuous vector field $e:\R^n\to\mathbb{S}^{n-1}$ yielded by Lemma \ref{lemma:campo}. We let $e_{1},\ldots,e_k:\R^n\to\mathbb{S}^{n-1}$ be continuous orthonormal vector fields spanning $V(\mu,x)$ at every $x\in K$ and we complete $\{e_1,\ldots,e_k,e\}$ to a basis of $\R^n$ of orthonormal continuous vector fields that we denote by $\{e_1,\ldots,e_k,e,e_{k+1},\ldots,e_{n-1}\}$.

Fix a ball $B(0,r)$ such that $K\subset B(0,r-1)$ and for any $\beta\in(0,1)$ we denote by $X_\beta$ the family of Lipschitz functions $f:B(0,r)\to\R$ such that 
\begin{equation}
    \lvert D_{e}f(x)\rvert\leq 1 \qquad \text{and}\qquad \lvert D_{e_j}f(x)\rvert\leq \beta\quad\text{for any $j=1,\ldots,n-1$},
    \label{condizioniXb}
\end{equation}
for $\Leb^n$-almost every $x\in\R^n$.
We metrize $X_\beta$ with the supremum norm and we note that this make $X_\beta$ a complete and separable metric space. Note also that $X_\beta$ is non-trivial as it contains all the $\beta$-Lipschitz functions. 
\end{definizione}

\begin{definizione}
For any $\beta>0$ and any $0\leq\sigma'<\sigma<1$ we can define on $X_\beta$ the functionals
\begin{align*}
T_{\sigma',\sigma}^+f:=\max\Big\{ \sup\Big\{\frac{f(x+v)-f(x)}{\lvert v\rvert}:\sigma'<\lvert v\rvert\leq \sigma\text{ and } x+v\in \supp(\mu)\setminus X(x,V(\mu,x),\sqrt{1-\alpha^2})\Big\},-n\Big\}\\
T_{\sigma',\sigma}^-f:=\min\Big\{ \inf\Big\{\frac{f(x+v)-f(x)}{\lvert v\rvert}:\sigma'<\lvert v\rvert\leq \sigma\text{ and } x+v\in \supp(\mu)\setminus X(x,V(\mu,x),\sqrt{1-\alpha^2})\Big\},n\Big\}.
\end{align*}
\end{definizione}

\begin{proposizione}
For any $0\leq \sigma'<\sigma<1$ the functionals
$$U_{\sigma',\sigma}^\pm f:=\int_K T_{\sigma',\sigma}^\pm f(z)d\mu(z),$$
are Baire class 1 on $X_\beta$. 
\end{proposizione}

\begin{proof}
As a first step we show that the  $T^+_{\sigma',\sigma}:X_\beta\to L^1 (\mu\trace K)$ are continuous whenever $0<\sigma'<\sigma<1$. The functions $T_{\sigma',\sigma}^+ f$ belong to $L^1(\mu\trace K)$ since $K$ has finite measure and $\lvert T_{\sigma',\sigma}^+ f\rvert\leq \mathrm{Lip}(f)+n$. In addition, it is immediate to see that:
$$\lvert T_{\sigma',\sigma}^+f(x)-T_{\sigma',\sigma}^+g(x)\rvert\leq \frac{2\lVert f-g\rVert_\infty}{\sigma'}\qquad \text{for $\mu$-almost every $x\in \R^n$},$$
thanks to the fact that if at some $x\in\R^n$ we have $(B(x,\sigma)\setminus B(x,\sigma'))\cap (\supp(\mu)\setminus X(x,V(\mu,x),\sqrt{1-\alpha^2}))=\emptyset$, then $T_{\sigma',\sigma}f(x)=-n$ for any $f\in X_\beta$. Integrating in $\mu$, we infer that:
$$\lVert T_{\sigma',\sigma}^+f(x)-T_{\sigma',\sigma}^+g(x)\rVert_{L^1(\mu\trace K)}\leq \frac{2\mu(K)}{\sigma'}\lVert f-g\rVert_\infty.$$
This implies in particular that $U^+_{\sigma',\sigma}$ is a continuous functional on $X_\beta$. Following verbatim the argument above, one can also prove the continuity of the functionals $T^-_{\sigma',\sigma}$.

In order to prove that $U_{0,\sigma}^\pm$ is of Baire class 1, thanks to \cite[Theorem 24.10]{Kechris} we just need to show that for any $f\in X_\beta$ we have:
\begin{equation}
    \lim_{j\to \infty}U_{j^{-1},\sigma}^\pm f=U_{0,\sigma}^\pm f.
    \label{convergenza}
\end{equation}
By dominated convergence theorem if we are able to show that $\lim_{j\to\infty}T_{j^{-1},\sigma}^\pm f(x)=T_{0,\sigma}^\pm f(x)$ for $\mu\trace K$-almost every $x\in B(0,r)$ then \eqref{convergenza} follows. This convergence however, follows elementary from the definition of $T_{\sigma',\sigma}^\pm f$. 
\end{proof}

\begin{proposizione}
Let $\beta< (8n^2)^{-1}\alpha$. Then for every $\sigma>0$ the continuity points of $U_{0,\sigma}^\pm$ are contained in the set
$$\mathcal{L}_\pm(\sigma):=\Big\{f\in X_\beta:\pm U_{0,\sigma}^\pm f\geq\frac{\alpha}{16n}\mu(K)\Big\}.$$
In particular both $\mathcal{L}_+(\sigma)$ and $\mathcal{L}_-(\sigma)$ are residual in $X_\beta$.
\end{proposizione}

\begin{proof}
We prove the result just for  $U^+_{0,\sigma}$. The argument to prove the analogous statement for $U_{0,\sigma}^-$ can be obtained following verbatim that for $U^+_{0,\sigma}$ while making few suitable changes of sign.

Assume by contradiction that $g$ is a continuity point for $U_{0,\sigma}^+$ contained in $X_\beta\setminus\mathcal{L}_+(\sigma)$. 
It is easy to see by convolution and rescaling that smooth functions are dense in $X_\beta$. Since $g$ is a continuity point for $U_{0,\sigma}^+$, for any $\ell\in\N$ we can find a smooth function $h_\ell\in X_\beta$ such that $\lVert g-h_\ell\rVert_\infty\leq 2^{-\ell}$ and $U_{0,\sigma}^+h_\ell\leq\alpha\mu(K)/8n$ and for \emph{any} $x\in \R^n$ we have
 $$\lvert D_{e}h_\ell(x)\rvert\leq 1 \qquad \text{and}\qquad \lvert D_{e_j}h_\ell(x)\rvert\leq \beta\quad\text{for any $j=1,\ldots,n-1$.}$$


Let $$A:=\{y\in K:T_{0,\sigma}^+h_\ell(y)\leq \alpha/8n\}.$$ Thanks to Besicovitch's covering theorem and \cite[Lemma 7.5]{AlbMar} we can cover $\mu$-almost all $A$ with countably many closed and disjoint balls $\{B(y_j,r_j)\}_{j\in\N}$ such that, for $0<\eta,\chi<(n2^{10\ell})^{-1}\beta^2$
\begin{itemize}
    \item[(i)] $r_j\leq 2^{-\ell}$, $\mu(A\cap B(y_j,r_j))\geq (1-\eta)\mu(B(y_j,r_j))$ and $\mu(\partial B(y_j,r_j))=0$,
    \item[(\hypertarget{ii}{ii})] for any $z\in B(y_j,r_j)$
    $$\lvert e(z)-e(y_j)\rvert+\lvert \nabla h_\ell(y_j)-\nabla h_\ell(z)\rvert+\Big\lvert\frac{h_\ell(z)-h_\ell(y_j)}{\lvert z-y_j\rvert}-\nabla h_\ell(z)\Big[\frac{z-y_j}{\lvert z-y_j\rvert}\Big]\Big\rvert\leq \chi^4,$$
    \item[(iii)] for any $j\in\N$ we can find $0<\rho_j<(n2^{\ell})^{-1}\beta^2$ and a compact subset $\tilde A_j$ of $A\cap B(y_j,(1-2\rho_j)r_j)$ such that $\mu(\tilde A_j)\geq (1-2\eta)\mu(B(y_j,r_j))$ and $\tilde A_j$ is $C(e(y_j),2^{-10\ell}\chi^2)$-null.
\end{itemize}
For any $j\in\N$ we let $\phi_j$ be a smooth $2(\rho_j r_j)^{-1}$-Lipschitz function such that $0\leq \phi_j\leq 1$, $\phi_j=1$ on $B(y_j,(1-\rho_j)r_j)$ and it is supported on $B(y_j,r_j)$.
Now fix $0<\varepsilon<\beta\chi^2$. Thanks to Proposition \ref{alb} we can find $\delta_j\leq 2^{-j}\rho_jr_j$ and a function $\omega_j$ such that: 
\begin{enumerate}
\item $0\leq \omega_j(x)\leq \varepsilon\beta\rho_j r_j$ for any $x\in\R^n$,
\item $\omega_j(x)\leq \omega_{j}(x+se(y_j))\leq \omega_{j}(x)+s$, for every $s>0$ and any $x\in\R^n$. Moreover, if the segment $[x,x+se(y_j)]$ is contained in $B(\tilde A_j,\delta_j)$, then
$\omega_j(x+se(y_j))=\omega_j(x)+s$,
\item $\lvert \omega_{j}(x+v)-\omega_{j}(x)\rvert\leq 2^{-9\ell}\chi^2\lvert v\rvert$, for every $v\in e(y_j)^{\perp}$,
\item$\omega_j$ is $1+2^{-9\ell}\chi^2$-Lipschitz.
\end{enumerate}
We thus define the function $g_\ell$ as
\begin{equation}\label{def:gell}
    g_\ell:=(1-2\chi)
    \Big(h_\ell+\sum_{j\in\N}[-
    \langle\nabla h_\ell(y_j), e(y_j)\rangle+1]\phi_j \omega_j\Big).
\end{equation}
First we estimate the supremum distance 
\begin{equation}
\begin{split}
     \lVert g-g_\ell\rVert_\infty\leq& \lVert g-h_\ell\rVert_\infty+2\chi\lVert h_\ell\rVert_\infty+(1-2\chi)\lVert h_\ell-(1-2\chi)^{-1}g_\ell\rVert_\infty\\
     \leq &2^{-\ell}+\chi(\lVert g\rVert_\infty+2^{-\ell})+(1-2\chi)\Big\lVert\sum_{j\in\N}(1-\langle\nabla h_\ell(y_j),e(y_j)\rangle)\Big\rVert_\infty\\
     \leq & 2^{-\ell}(2+\lVert g\rVert_\infty+(1+(n-1)\beta^2)^{1/2})\leq 2^{-\ell}(4+\lVert g\rVert_\infty),
\end{split}
\end{equation}
where the last inequality follows from the choice of $\beta$.
The above computation shows that the sequence $g_\ell$ converges in the supremum distance.

Let us now prove that $g_\ell\in X_\beta$. If $z\not\in \cup_j B(y_j,r_j)=:\mathcal{I}$ then the functions $h_\ell$ and $g_\ell$ and their gradients coincide at $z$ and hence $g_\ell$ satisfy \eqref{condizioniXb} on $\mathcal{I}^c$. If on the other hand $z\in \mathcal{I}$, there exists a unique $j\in\N$ such that $z\in B(y_j,r_j)$. In particular, differentiating \eqref{def:gell} we get
\begin{equation}
    \begin{split}
\nabla g_\ell(z)=(1-2\chi)\Big[\nabla h_\ell(z)+&[-\langle\nabla h_\ell(y_j), e(y_j)\rangle+1]\nabla \phi_j(z)\omega_j(z)
+[-\langle\nabla h_\ell(y_j), e(y_j)\rangle+1]\phi_j(z)\nabla\omega_j(z)\Big].
\nonumber
    \end{split}
\end{equation}
So that, for $\Leb^n$-almost every $x\in\R^n$ we have
\begin{equation}
    \begin{split}
      & \lvert\langle \nabla g_\ell(z),e(z)\rangle\rvert\leq(1-2\chi)\Big\lvert\langle \nabla h_\ell(z),e(z)\rangle+[-\langle\nabla h_\ell(y_j), e(y_j)\rangle+1]\phi_j(z)\langle\nabla\omega_j(z),e(z)\rangle \Big\rvert+4\varepsilon\beta,
      \nonumber
     \end{split}
\end{equation}
where in the estimate above we have used the fact that $\lvert-\langle\nabla h_\ell(y_j), e(y_j)\rangle+1\rvert\leq 2$, $\lVert \nabla\phi\rVert_{L^\infty(\Leb^n)}\leq 2(\rho_jr_j)^{-1}$ and $\lVert \omega_j\rVert_\infty\leq \varepsilon\beta\rho_j r_j$. Now we replace $z$ with $y_j$ in the first addendum, by means of the estimate (\hyperlink{ii}{ii}), obtaining
\begin{equation}
    \begin{split}
     \lvert\langle \nabla g_\ell(z),e(z)\rangle\rvert  \leq 3(1-2\chi)\chi^2+(1-2\chi)\Big\lvert\langle\nabla h_\ell(y_j), e(y_j)\rangle\big(1-\phi_j(z)\langle\nabla\omega_j(z),e(z)\rangle\big)+\phi_j(z)\langle\nabla\omega_j(z),e(z)\rangle \Big\rvert+2\varepsilon\beta.
       \nonumber
       \end{split}
       \end{equation}
       Finally, substituting $z$ with $y_j$ in the argument of the vector field $e$ we deduce thanks to (\hyperlink{ii}{ii}) that
       \begin{equation}
    \begin{split}
       \lvert\langle \nabla g_\ell(z),e(z)\rangle\rvert \leq& 3(1-2\chi)\chi^2+2\varepsilon\beta+6(1-2\chi)(1+2^{-9\ell}\chi)\chi^2\\
       &\qquad\qquad\quad\,\,+(1-2\chi)\big\lvert\langle\nabla h_\ell(y_j), e(y_j)\rangle\big(1-\phi_j(z)\langle\nabla\omega_j(z),e(y_j)\rangle\big)+\phi_j(z)\langle\nabla\omega_j(z),e(y_j)\rangle \big\rvert\\
       \leq &3(1-2\chi)\chi^2+2\varepsilon\beta+6(1-2\chi)(1+2^{-9\ell}\chi)\chi^2+(1-2\chi)\leq 1,
       \nonumber
     \end{split}
\end{equation}
where the the last inequality follows from the choice of $\chi,\beta,\varepsilon$. Furthermore, for any $q  =1,\ldots,n-1$ we infer similarly that:
\begin{equation}
    \begin{split}
       & \lvert g_\ell(z+t e_q(z))-g_\ell(z)\rvert \leq (1-2\chi)\big\lvert h_\ell(z+ t\,e_q(z))-h_\ell(z)\big\rvert\\
        &\qquad\qquad\qquad\qquad\qquad\qquad+(1-2\chi)\lvert [1-\langle\nabla h_\ell(y_j),e(y_j)\rangle](\phi_j(z+t\,e_q(z))-\phi_j(z))\omega_j(z)\rvert\\
        &\qquad\qquad\qquad\qquad\qquad\qquad+(1-2\chi)\lvert[1-\langle\nabla h_\ell(y_j),e(y_j)\rangle]\phi_j(z)(\omega_j(z+t\,e_q(y_j))-\omega_j(z))\rvert\\
         &\qquad\qquad\qquad\qquad\qquad\qquad+(1-2\chi)\lvert[1-\langle\nabla h_\ell(y_j),e(y_j)\rangle]\phi_j(z)(\omega_j(z+t\,e_q(z))-\omega_j(z+t\,e_q(y_j)))\rvert+o(\lvert t\rvert)\\
         \leq &(1-2\chi)\beta \lvert t\rvert+4(1-2\chi)(\beta\varepsilon\rho_jr_j)(\rho_jr_j)^{-1}\lvert t\rvert+3\cdot2^{-9\ell}(1-2\chi)\chi^2 \lvert t\rvert+3(1-2\chi)(1+2^{-9\ell}\chi)\chi^4\lvert t\rvert+o(\lvert t\rvert)\\
         \leq& (1-2\chi)(\beta+4\beta\varepsilon+4\cdot 2^{-9\ell} \chi^2+4(1+2^{-9\ell}\chi)\chi^4) \lvert t\rvert\leq (1-2\chi)(1+10\chi^2)\beta \lvert t\rvert+o(\lvert t\rvert)< \beta \lvert t\rvert,
         \nonumber
    \end{split}
\end{equation}
provided $|t|$ is chosen sufficiently small (depending on $z$) and
where the second to last inequality holds thanks to the choice of $\chi,\varepsilon$ and for $\ell$ sufficiently big, in such a way that $2^{-\ell}\leq \beta$.
The above bound implies that in particular
\begin{equation}
\begin{split}
  \lvert \langle \nabla g_\ell(z),e_q(z)\rangle\rvert\leq \beta\text{ for $\Leb^n$-almost every $x\in\R^n$.}
\end{split}
\end{equation}
This concludes the proof that for $\ell$ sufficiently big we have that $g_\ell\in X_\beta$.

The next step in the proof is to show that the functions $g_\ell$ satisfy the inequality $U_{0,\sigma}^+g_\ell \geq \alpha\mu(K)/8$ for $\ell$ sufficiently big, and this contradicts the continuity of $U_{0,\sigma}^+$ at $g$. In order to see this, we first estimate from below the partial derivative of $g_\ell$ along $e$ on the points of $\tilde{A}_j$ for any $j$. So, let us fix for any $j\in\N$ a point $z\in \tilde A_j$. Then, let $0<\lambda_0<\delta_j$ be so small that $\phi_j(z+\lambda e(z))=1$ for any $0<\lambda<\lambda_0$ and note that
\begin{equation}
\begin{split}
      \langle g_\ell(z+\lambda& e(z))-g_\ell(z),e(z)\rangle\geq (1-2\chi)\Big[(h_\ell(z+\lambda e(z))-h_\ell(z))
  +[1-\langle\nabla h_\ell(y_j),e(y_j)\rangle](\omega_j(z+\lambda e(z))-\omega_j(z))\Big]\\
  \geq&  (1-2\chi)\big[-\chi^2\lambda+\lambda\langle\nabla h_\ell(z),e(z)\rangle+[1-\langle\nabla h_\ell(y_j),e(y_j)\rangle]\lambda\big]
  \geq \lambda(1-2\chi)(1-4\chi^2)\geq (1-6\chi)\lambda.
    \nonumber
\end{split}
\end{equation}
This implies in particular that for any unitary vector $v\in C(e(z),(n-k)^{-1}\alpha)$, for any $\lambda>0$ we have
\begin{equation}
\begin{split}
       g_\ell(z+\lambda v)-g_\ell(z)\geq& g_\ell (z+\lambda v)-g_\ell(z+\lambda \langle e(z),v\rangle e(z))+ g_\ell(z+\lambda \langle e(z),v\rangle e(z))-g_\ell(z)\\
       \geq&\alpha(n-k)^{-1}(1-6\chi)\lambda-\beta\sqrt{n-1}\lambda\geq (\alpha/2(n-k))\lambda-\beta n\lambda>\alpha \lambda/4(n-k),
\end{split}
\end{equation}
where the last inequality follows from the choice of $\beta$. However, thanks to choice of $K$, see \eqref{e:conipieni}, we infer that $$T_{0,\sigma}^+g_\ell(z)\geq \alpha/4(n-k)\qquad\text{ for any $z\in \cup_j \tilde{A}_j$.}$$ This allows us to infer that
 \begin{equation}
 \begin{split}
     U_{0,\sigma}^+ g_\ell=&\int_AT^+_{0,\sigma}g_\ell d\mu+\int_{K\setminus A}T^+_{0,\sigma}g_\ell d\mu\geq\int_AT^+_{0,\sigma}g_\ell d\mu+ \alpha\mu(K\setminus A)\\
     = &\int_{A\setminus \cup_j     \tilde A_j}  T^+_{0,\sigma}g_\ell d\mu+\sum_{j\in\N}\int _{A_j}T^+_{0,\sigma}g_\ell d\mu+ \alpha\mu(K\setminus A)\\
     \geq& -\mu(A\setminus \bigcup_{j\in\N} A_j)\mathrm{Lip(g_\ell)}+\frac{\alpha}{4(n-k)}\mu( \bigcup_{j\in\N} A_j)+\alpha\mu(K\setminus A)\\
     \geq &-2\mu(A\setminus \bigcup_{j\in\N} A_j)+\frac{\alpha}{4(n-k)}\mu\Big((K\setminus A)\cup \bigcup_{j\in\N} A_j\Big)\\
     \geq& -4\eta\mu(K)+\frac{\alpha}{4(n-k)}(1-2\eta)\mu(K)\geq \frac{\alpha}{8n}\mu(K).
  \nonumber
 \end{split}
\end{equation}
for $\ell$ sufficiently big.\\ 

Since the functional $U^+_{0,\sigma}$ is of Baire class 1, thanks to  \cite[Chapter 7]{Oxtobi} we know that the set of the continuity points of $U^+_{0,\sigma}$ is residual. However, since thanks to the above argument $\mathcal{L}_+(\sigma)$ contains the continuity points of $U^+_{0,\sigma}$, we conclude that $\mathcal{L}_+(\sigma)$ is residual in $X_\beta$.
\end{proof}

\begin{proof}[Proof of Proposition \ref{coroll}]
Let $\beta:=(16n^2)^{-1}\alpha$ and let $\mathfrak{c}(\alpha):=\alpha /16n$ note that since the countable intersection of residual sets is residual, we can find a Lipschitz function $f$ in $X_\beta$ such that $f\in\cap _{\sigma\in \mathbb{Q}\cap (0,1)}(\mathcal{L}_+(\sigma)\cap \mathcal{L}_-(\sigma))$. In particular, for any $\sigma>0$ we have
$$U_{0,\sigma}^-f\leq -\mathfrak{c}(\alpha)\mu(K)<\mathfrak{c}(\alpha)\mu(K)\leq U_{0,\sigma}^+f.$$
In particular, defined $\Delta T_\sigma f(z):=T_{0,\sigma}^+f(z)-T_{0,\sigma}^-f(z)$ and $C_\sigma:=\{z\in K:\Delta T_\sigma(z)>\mathfrak{c}(\alpha)\}$
\begin{equation}
    2\mathfrak{c}(\alpha)\mu(K)\leq \int_K \Delta T_\sigma (z)d\mu\trace K(z)\leq \mu(K\setminus C_\sigma)\mathfrak{c}(\alpha)+2\mathrm{Lip}(f)\mu(C_\sigma).
    \nonumber
\end{equation}
Thanks to the above computation we infer in particular that $\mu(C_\sigma)\geq \mathfrak{c}(\alpha)\mu(K)/2\mathrm{Lip}(f)$ for any $\sigma>0$. Thus, defined $E:=\bigcap_{j\in\N}\bigcup_{l\geq j}C_{1/l}$, Fatou's Lemma implies that:
$$\frac{\mathfrak{c}(\alpha)\mu(K)}{2\mathrm{Lip}(f)}\leq \limsup_{p\to \infty} \mu(C_{1/p})\leq \int\limsup_{p\to \infty}\mathbb{1}_{C_{1/p}}d\mu=\mu(E),$$
where $\mathbb{1}_{C_{1/p}}$ denotes the indicator function of the set $C_{1/p}$. Therefore, $E$ is a Borel set of positive $\mu$-measure such that for $\mu$-almost every $z\in E$ there exists a sequence of natural numbers (depending on $z$) such that $p\to\infty$ and $\Delta T_{1/p}>\mathfrak{c}(\alpha)$. In particular, for $\mu$-almost every $z\in E$ we have:
\begin{equation}
  \mathfrak{c}(\alpha) < \liminf_{p\to\infty}(T^+_{0,1/p}f(z)-T^-_{0,1/p}f(z))=\lim_{p\to\infty}(T^+_{0,1/p}f(z)-T^-_{0,1/p}f(z)),
\end{equation}
where the last identity comes from the fact that $p\mapsto T^+_{0,1/p}f(z)$ is decreasing and  $p\mapsto T^-_{0,1/p}f(z)$ is increasing for any $z$. However, thanks to the definition of $T^+_{0,1/p}f$ and $T^-_{0,1/p}f$ it is immediate to see that for $\mu$-almost every $z\in E$ we can find a sequence $y_i=y_i(z)\in \supp(\mu)\cap C(e(z),\alpha)\cap B(z,i^{-1})$ such that
    $$\frac{y_i-z}{\lvert y_i-z\rvert}\to v\qquad \text{and}\qquad \limsup_{i\to \infty}\frac{f(y_i)-f(z)}{\lvert y_i-z\rvert}-\liminf_{i\to \infty}\frac{f(y_i)-f(z)}{\lvert y_i-z\rvert}>\frac{\mathfrak{c}(\alpha)}{2}.$$
\end{proof}
\section{Proof of Theorem \ref{t:main}}
Without loss of generality we can restrict our attention to finite measures. Assume that $\mu$ is a finite sum of rectifiable measures. For every $\varepsilon>0$ there exist finitely many disjoint, compact submanifolds $S_j$ for $(j=1,\dots, N)$ of class $C^1$ (of any dimension between $0$ and $n$) such that  denoting $K:=\bigcup_{j=1}^N S_j$ it holds $\mu(\R^n\setminus K)<\varepsilon/2$. Consider now any Lipschitz function $f:\R^n\to\R$. By \cite[Theorem 1.1 (i)]{AlbMar} and Lusin's theorem, we can find a closed subset $C\subset K$ such that $\mu(K\setminus C)<\varepsilon/2$ and for every $x\in C$ the derivative $d_{V(\mu,x)}f(x)$, see \cite[\S 2.1]{AlbMar}, exists and is continuous. Let $d:C\to\R^n$ be obtained extending $d_{V(\mu,\cdot)}f$ to be zero in the directions orthogonal to $V(\mu,\cdot)$. By \cite[Proposition 2.9 (iii)]{AlbMar} and since the $S_j$'s have positive mutual distances, we can apply Whitney's extension theorem, see \cite[Theorem 6.10]{EvansGar}, deducing that there exists a function $g:\R^n\to\R$ of class $C^1$ such that $g=f$ and $dg=d$ on $C$. Hence Lipschtz functions admit a Lusin type approximation with respect to $\mu$ with functions of class $C^1$.

Assume now that $\mu$ is not a finite sum of rectifiable measures and write $\mu=\sum_{k=0}^n\mu\trace E_k$, where $E_k:=\{x\in\R^n:\dim(V(\mu,x))=k\}$. Then there exists $k\in \{0,\dots, n-1\}$ such that $\mu\trace E_k$ is not a $k$-rectifiable measure: the case $k=n$ can be excluded by combining \cite{AlbMar} and \cite{DPR}, so to ensure that a measure on $\R^n$ whose decomposability bundle has dimension $n$ is absolutely continuous with respect to the Lebesgue measure $\Leb^n$. Let $\nu$ be the supremum of all $k$-rectifiable measures $\sigma\leq\mu\trace E_k$ and let $E$ be any Borel set such that $\nu=\mu\trace (\R^n\setminus E)$. We claim that $\mu\trace E$ satisfies the assumptions of Lemma \ref{lemma:coni}. To prove the claim, consider a $k$-dimensional surface $S$ of class $C^1$ and assume by contradiction that $\eta:=\mu\trace (E\cap S)$ is non zero. If $G=\{\mu_t\}_{t\in I}\in\mathscr{F}_\eta$ is a family as in \cite[Proposition 2.8 (ii)]{AlbMar}, then $\supp(\mu_t)\subset S$ for almost every $t\in I$. This implies that $V(\eta,x)={\mathrm{Tan}(S,x)}$ for $\eta$-almost every $x$. Fix now a point $y\in\supp(\eta)$ and observe that, denoting $W:=\mathrm{Tan}(S,y)$, the family $\{(p_W)_\sharp\mu_t\}_{t\in I}$ belongs to $\mathscr{F}_{(p_W)_\sharp\eta}$ and that $V((p_W)_\sharp\eta, \cdot)$ is $k$-dimensional in a neighbourhood of $p_W(y)$. By \cite[Corollary 1.12]{DPR}, the Radon-Nikodym derivative of $(p_W)_\sharp\eta$ with respect to $\Haus^k\trace W$ is positive and finite in a neighbourhood of $p_W(y)$. Since $p_W$ is bi-Lipschitz from $S$ to $W$ in a (relative) neighbourhood of $y$, this implies that $\eta$ has a non trivial absolutely continuous part with respect to $\Haus^k\trace S$, which contradicts the maximality of $\sigma$. Hence, $\mu\trace E$ satisfies the assumptions of Lemma \ref{lemma:coni}.

Let $f:\R^n\to\R$ be the Lipschitz function obtained from Proposition \ref{coroll}. Clearly there exists no function $g:\R^n\to\R$ of class $C^1$ which coincides with $f$ on a set of positive measure with respect to $\mu\trace E$, hence Lipschtz functions do not admit a Lusin type approximation with respect to $\mu$ with functions of class $C^1$.

\printbibliography

\section*{Acknowledgments}
A. Ma. acknowledges partial support from 
PRIN 2017TEXA3H\_002 \emph{Gradient flows, Optimal Transport and Metric Measure Structures}. A. Me. is supported by the Simons Foundation grant 601941, GD.
\end{document}